\documentclass[english]{psp}

\usepackage[T1]{fontenc}
\usepackage[latin9]{inputenc}
\usepackage{graphicx}
\usepackage{amsfonts}

\makeatletter
\def \real {\mathbb R}
\newtheorem{thm}{theorem}
\newnumbered{example}{example}
\newtheorem{defn}{definition}
\newtheorem{lem}{lemma}
\newtheorem{cor}{corollary}

\usepackage{babel}
\makeatother

\begin{document}

\title{Discontinuities of plane functions projected from a surface with
methods for finding these}

\author{Burzin Bhavnagri }

\begin{abstract}
A result is given to find points where a real valued function on the
plane is not smooth. Provided this function is induced by a smooth
mapping from three dimensions to the plane, from a function on surfaces
in three dimensions. This has applications to numerical methods such
as image processing. 
\end{abstract}


\maketitle

\section{introduction}

A function $f:{\bf \real}^{m}\rightarrow{\bf \real}$ is continuous
at $x\in\real^{m}$ if for every sequence $x_{i}$ converging to $x$,
the sequence $f(x_{i})$ converges to $f(x)$. Thus any subsequence
of any  sequence $f(x_{i})$ converges to the same limit, namely $f(x)$.

Unfortunately, the definition of continuity is not helpful as to how
we may locate a discontinuity in an image, which is a finite array
of measurements. This paper presents a mathematical result which enables
one to find points where a function $I:\real^{2}\rightarrow\real$
is not smooth, such as discontinuities. It is assumed that $I$ is
induced by a smooth mapping \[
\iota:\real^{3}\rightarrow\real^{2}\quad\iota(x,y,z)=(\frac{x}{z},\frac{y}{z})\qquad z\neq0\]

In previous work \cite{le-bhavnagri97} a problem which arose in computer
vision was formulated using the shape spaces introduced by Kendall.
An application of the explicit solution was the following. An image
is acquired from a video camera. The image contains a shape such as
a triangle, square or circle. The problem is to classify it into categories
such as triangles, squares or circles. For example, an image of a
triangle will typically contain far more than three points, so it
needs to be simplified. A square on the other hand should not be simplified
to a triangle, and a circle should not be simplified very much at
all. In order to obtain the points to be simplified in the first place,
we first have to find discontinuities in $I$. 

This paper is concerned with the mathematical aspects of the discontinuity
identification problem. In what follows it is assumed the reader is
familiar with differential geometry. Although the result is motivated
by computer vision and human vision, differential geometry is not
familiar there. As a mathematical result it could be extended by a
wider scientific audience by replacing the above $\iota$ with some
other measurement $\iota$. 

It is very surprising that using optical principles, we can establish
mathematically that there really is a discontinuity.  Not some other
property such as a local maximum in a derivative, but a discontinuity.
The derivative is defined in calculus or real analysis. All books
on real analysis \cite{binmore82} make it clear that continuity is
a necessary condition for the derivative to be well-defined. 

When grappling with this problem others have used piecewise smooth
or other assumptions. Instead here the contrapositive will be used
to obtain a more general result. Owing to the microstructure of a
scene surface, the light radiated from the surface undergoes extremely
rapid intensity changes over small regions. This can be confirmed
by examining gray values of neighbouring points in a real image. It
is for this reason that this more general result may be helpful.

A fundamental psychological observation about the nature of vision
was made by Wertheimer, who noticed the apparent motion not of individual
dots, but of fields in images presented sequentially as a movie \cite{wertheimer12}.
This started the Gestalt school of psychology \cite{wertheimer38,kanizsa79}. 

A differential operator maps a greyscale image to a vector field.
It is implicitly suggested by Hoffman \cite{hoffman66} that the primary
visual cortex of the brain produces some sort of vector field. It
should be possible to line up the vectors in a head to tail fashion
to form contours. We want to know when this results in a family of
closed oriented contours rather like a topographic map. The answer
is provided by a famous theorem of differential geometry, namely the
Foliation theorem of Frobenius.

\section{Sufficient condition for discontinuity}

In what follows we consider a real valued function on the plane. Some
examples of such functions in a color image are red, blue, green,
hue, saturation or brightness. 

\begin{example}
Suppose $I$ is a smooth function (intensity). Let $M$ be the subset
of the plane on which $(-\partial I/\partial y,\partial I/\partial x)$
is non-zero. $M$ is an open set because by definition of continuity
the inverse of a continuous function maps open sets to open sets,
and the line minus zero is open. The map \[
(x,y)\mapsto{\rm span}(-\,\frac{\partial I}{\partial y}(x,y),\frac{\partial I}{\partial x}(x,y))\]
 is a smooth function that maps $M$ to the set of lines through the
origin in the plane ${\bf P}^{1}$ (projective space). 
\end{example}
It is an example of what is called a one dimensional distribution\index{distribution}
in differential geometry \cite{warner83}, \cite{brickell-clark70}.
The vector field is called a basis for this distribution.

\begin{example}
Consider $f(x,y)=x+y$ and $g(x,y)=e^{x+y}$. \[
-\frac{\partial f}{\partial y}=-1\quad-\frac{\partial g}{\partial y}=-e^{x+y\quad}\frac{\partial f}{\partial x}=-1\quad\frac{\partial g}{\partial x}=-e^{x+y}\]
 Thus at every point in the plane, the Hamiltonian vector of $g$
is $e^{x+y}$ times the Hamiltonian vector of $f$. These are two
different vector fields that are a basis for the same distribution.
On the other hand the Hamiltonian vector of $e^{x}+e^{y}$ is not
collinear with the Hamiltonian of $f$, so not every pair of vector
fields does span a one dimensional distribution.
\end{example}
\begin{defn}
Suppose $Y$ is a vector field, and it is the basis for a distribution.
The distribution is called integrable if at every point one can find
a disk on which a coordinate system $(x_{1},x_{2})$ can be chosen
such that \[
Y_{1}\frac{\partial}{\partial y_{1}}+Y_{2}\frac{\partial}{\partial y_{2}}=\frac{\partial}{\partial x_{1}}\]

The following is based on proposition 1.53 in \cite{warner83}.
\end{defn}
\begin{lem}
Let $m\in M,$ and let $X$ be a smooth vector field on $M$ such
that $X(m)\neq0$. Then there exists a coordinate system $(U,\varphi)$
with coordinate functions $x_{1},\, x_{2}$ on a neighborhood of $m$
such that 

\begin{equation}
X\mid U=\frac{\partial}{\partial x_{1}}\mid U\end{equation}

\end{lem}
\begin{proof}
Choose a coordinate system $(V,\tau)$ centered at m with coordinate
functions $y_{1},\, y_{2}$ , such that

\begin{equation}
X_{m}=\frac{\partial}{\partial y_{1}}\mid_{m}.\end{equation}

It follows from Picard's theorem that there exists an $\varepsilon>0$
and a neighborhood $W$ of the origin in $\real$ such that the map
\[
\sigma(t,a_{2})=X_{t}(\tau^{-1}(0,a_{2}))\]
is well-defined and smooth for $(t,a_{2})$$\in(-\varepsilon,\varepsilon)\times W\subset\real^{2}$. 

Now, $\sigma$ is non-singular at the origin since \[
d\sigma\left(\frac{\partial}{\partial r_{i}}\right)\mid_{0}=X_{m}=\frac{\partial}{\partial y_{i}}\mid_{m}\quad(i\geq1)\]
 Thus by the Inverse Function Theorem, $\varphi=\sigma^{-1}$is a
coordinate map on some neighborhood $U$ of $m$. Let $x_{1},x_{2}$denote
the coordinate functions of the coordinate system $(U,\varphi)$.
Then since

\[
d\sigma\left(\frac{\partial}{\partial r_{1}}\mid_{(t,a_{2})}\right)=X_{\sigma(t,a_{2})}\]
we have \[
X\mid U=\frac{\partial}{\partial x_{1}}\mid U\]
.
\end{proof}
\begin{lem}
\label{lem:integrability}Any one dimensional distribution is integrable.
\end{lem}
\begin{proof}
See proposition 11.1.1 in \cite{brickell-clark70} or proposition
1.53 in \cite{warner83}. 
\end{proof}
\begin{defn}
Let $(a,b)$ be an interval on the real line, but $a$ can be $-\infty$
and $b$ can be $+\infty$. A differentiable curve $\gamma:(a,b)\rightarrow\real^{2}$
is an integral curve\index{integral curve} of a vector field $X$
if $\gamma^{\prime}(t)=X(\gamma(t))$ for all $t\in(0,1)$ where $\gamma^{\prime}$
denotes the derivative of $\gamma$.
\end{defn}
\label{IntegralCurves}%
\begin{figure}

\caption{(a) Vector field X on surface patch (b) integral curves of X}
(a)\includegraphics{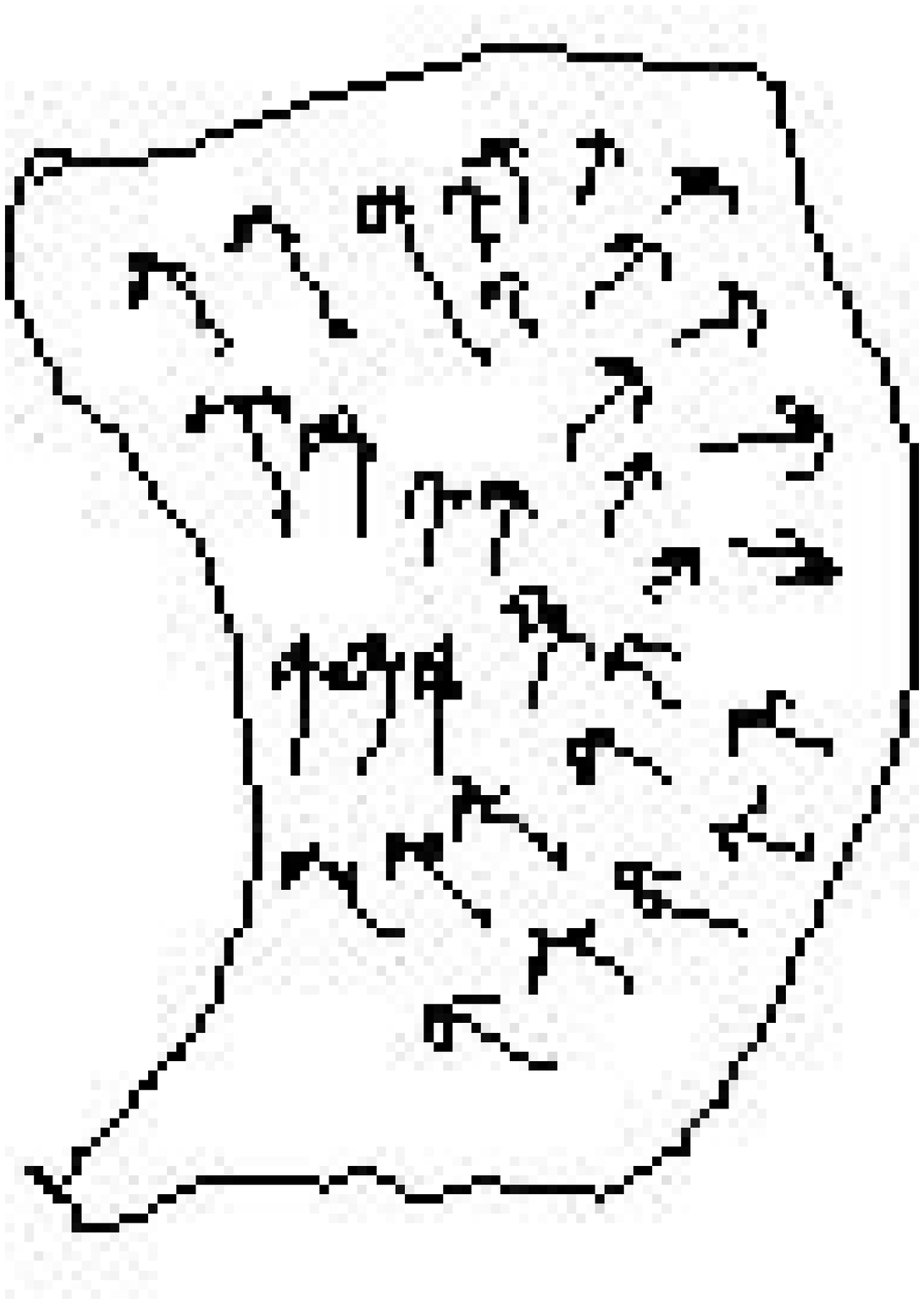}(b)\includegraphics{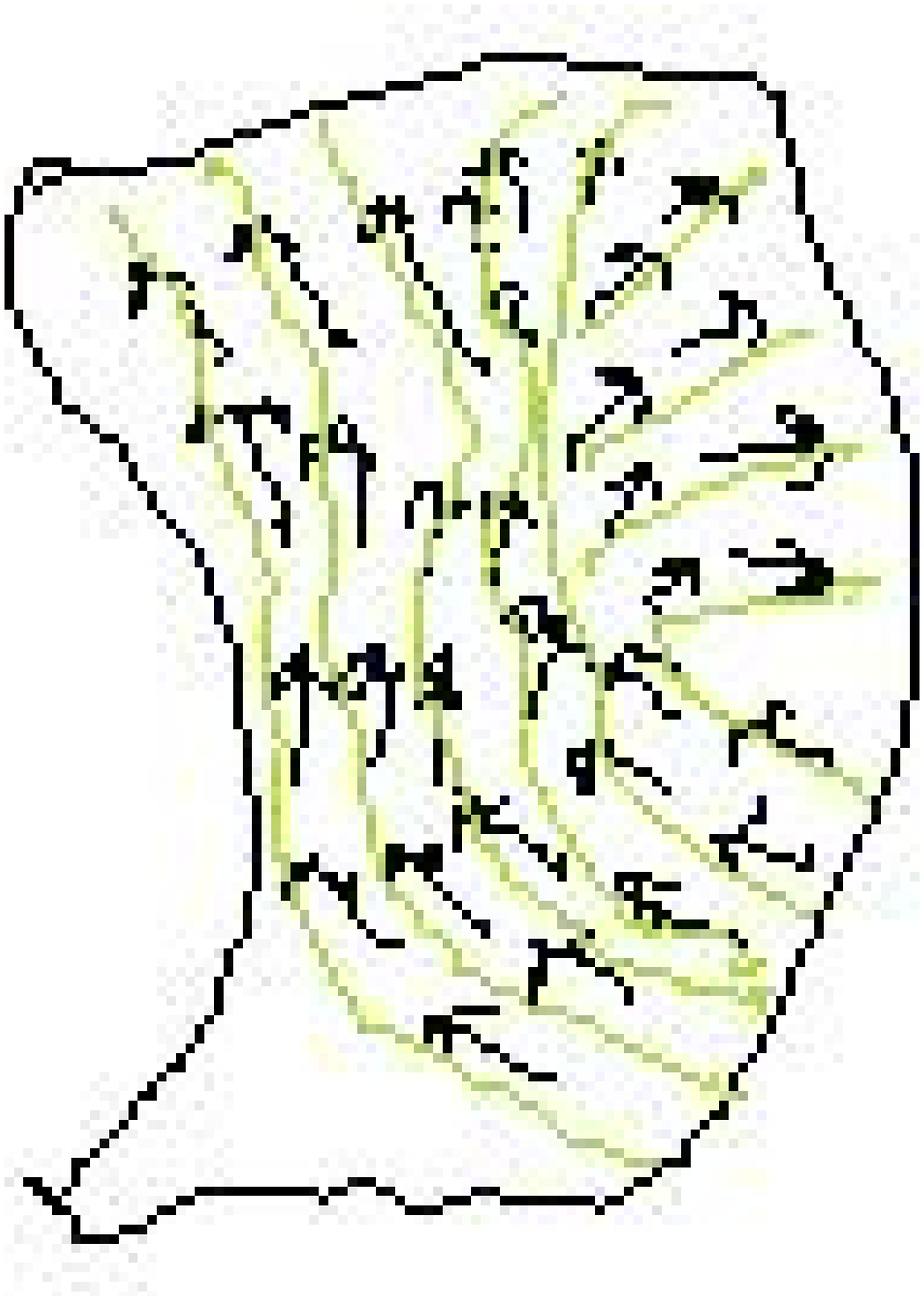}

\end{figure}
Thus an integral curve is a curve whose tangent at each point coincides
with the vector of the vector field at the same point. Since the gradient
vector points in the direction of steepest descent, perpendicular
vectors are produced by the Hamiltonian operator, and if we join these
Hamiltonian vectors in a head to tail fashion, they line up into integral
curves.

\begin{defn}
An integral curve is called a complete integral curve if its domain
is $\real=(-\infty,+\infty)$.
\end{defn}
The following is the authors Theorem 3.3 from \cite{bhav96}. 

\begin{thm}
\label{thm:foliationthm}Let $X$ be a smooth vector field on the
plane, and $M$ the subset of the plane on which $X$ is non-zero.
Each point of $M$ lies in precisely one maximal integral curve, and
there is a one parameter family of such integral curves in $M$. An
integral curve does not intersect itself.
\end{thm}
\begin{proof}
By lemma \ref{lem:integrability} the distribution $m\mapsto{\rm span}(X(m))$
is integrable. Proposition 11.2.1 in \cite{brickell-clark70} asserts
that an integrable distribution is involutive. A one dimensional integral
submanifold of $M$ is a smooth curve in $M$ whose tangent equals
the line in the distribution at each point of the curve. The Frobenius
theorem (theorem 1.60) in \cite{warner83} asserts that there is a
cubic coordinate system centred at each point of $M$, with coordinates
$x_{1},x_{2}$ such that $x_{2}={\rm constant}$ are integral submanifolds
of $M$. By theorem 1.64 in \cite{warner83}, there is a unique maximal
integral submanifold passing through each point of $M$. By proposition
11.3.1 in \cite{brickell-clark70} any integral curve of $X$ is a
integral submanifold of $M$. Consequently each point of $M$ lies
in precisely one complete integral curve, and there is a one parameter
family of such curves in $M$. 

\end{proof}
\label{ProjectedIntegralCurves}%
\begin{figure}

\caption{Integral curves on plane corresponding to smooth surface integral
curves; lines denote $\iota$}
\includegraphics{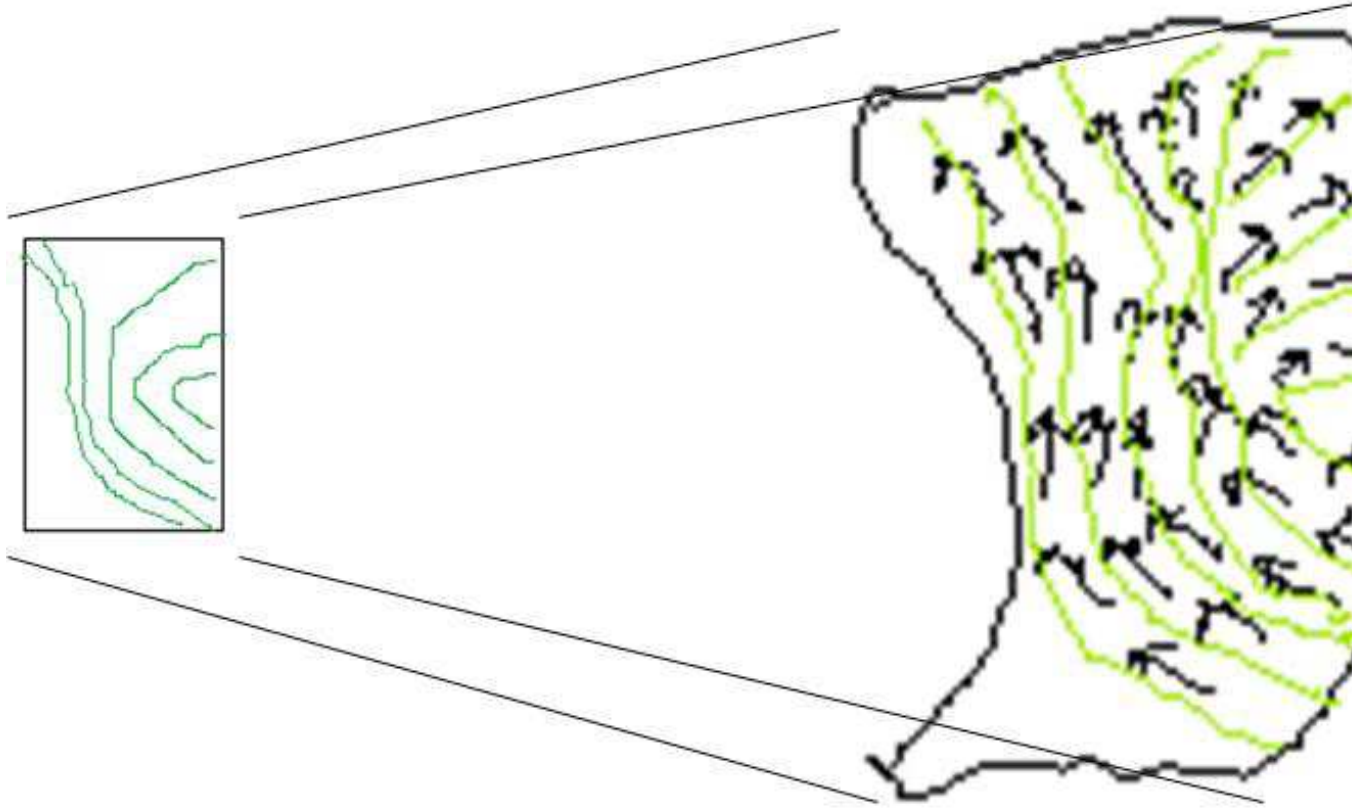}

\end{figure}

\begin{example}
Consider the circle $(z-1)^{2}+x^{2}=\frac{1}{2}$. Let $\iota:(x,z)\mapsto x/z$
produce a one-dimensional image of this circle. At the point $x=\frac{1}{2},z=\frac{1}{2}$
the tangent to the circle passes through the origin, so this is an
occluding point of the image. Let $t$ be the image coordinate, and
calculate the differential \[
d\iota(\alpha\frac{\partial}{\partial x}+\beta\frac{\partial}{\partial z})=(1/z-x/z^{2})\frac{\partial}{\partial t}\]
 So the tangent vector to the circle at $(\frac{1}{2},\frac{1}{2})$
is \[
v=\frac{1}{2}\frac{\partial}{\partial x}+\frac{1}{2}\frac{\partial}{\partial z}\]
 Evaluating $d\iota(v)$ we find it is zero. Since $\alpha=x,\beta=z$
at any occluding point, $d\iota(v)=0$ at any occluding point.
\end{example}
\begin{defn}
Let $\iota:(x,y,z)\rightarrow(t,s)$ where $t=x/z,\, s=y/z$, and
$z\neq0$. 
\end{defn}
Then \[
d\iota(v_{1}\frac{\partial}{\partial x}+v_{2}\frac{\partial}{\partial y}+v_{3}\frac{\partial}{\partial z})\mid_{a}=(\frac{v_{1}}{z}-\frac{v_{3}x}{z^{2}})\frac{\partial}{\partial t}+(\frac{v_{2}}{z}-\frac{v_{3}y}{z^{2}})\frac{\partial}{\partial s}\]
 Thus $d\iota(v)\mid_{a}=0$ if and only if $v$ is a multiple of
$a$. Thus $d\iota(v)$ has rank $2$ at all points on a surface disjoint
from the focal plane, except occluding points, where it has rank $1$.

\section{Sufficient condition for a discontinuity}

The corollaries to the theorem \ref{thm:foliationthm} give us a condition
to establish that there is a discontinuity or occlusion in an image.

\begin{cor}
Suppose $M_{1}$ is a smooth two dimensional manifold in the scene,
whose tangent planes are everywhere disjoint from the optical centre
of a camera. Suppose also that the intensity on \textup{$M_{1}$}
is smooth.

(i) Then there is a smooth invertible mapping between vector fields
on \textup{$M_{1}$}and vector fields on the image of \textup{$M_{1}$}.

(ii) Moreover, the non-vanishing vector fields on \textup{$M_{1}$}
correspond to the non-vanishing vector fields on the image of \textup{$M_{1}$.}

\textup{(iii) }Each point in a smooth and unoccluded image of a smooth
and unoccluded surface lies in precisely one maximal integral curve
of a smooth image vector field, and there is a one parameter family
of such integral curves. 
\end{cor}
\begin{proof}
(i) The differential of $\iota$ maps a vector field on a smooth surface
to a vector field on its image. Vector fields on images minus the
occluding points can be mapped to vector fields on a smooth surface
by the inverse differential. 

(ii) That non-vanishing vector fields on $M_{1}$ correspond to the
non-vanishing vector fields on the image of $M_{1}$, follows because
the tangent planes were disjoint from the optical centre. 

(iii) From Theorem \ref{thm:foliationthm}, each point in a smooth
and unoccluded image of a smooth and unoccluded surface lies in precisely
one maximal integral curve of a smooth image vector field, and there
is a one parameter family of such integral curves.
\end{proof}
The contrapositive is often used in mathematics, mainly for proofs
by contradiction, such as the irrationality of the square root of
two. Consider any predicates $P,\: Q$ and denote $notP$ and $notQ$
by $\neg P,\neg Q$. The contrapositive of $P$$\Rightarrow Q$ is
$\neg Q\Rightarrow\neg P$. Since $P\Rightarrow Q$ is $\neg P$$\vee Q$
then $\neg(\neg P)$$\vee\neg Q$ which is $P\vee\neg Q$ thus $\neg Q\Rightarrow\neg P$.
It is the contrapositive that enables us to find a discontinuity. 

\begin{cor}
\label{cor:contra-foliation}If an image point lies in more than one
maximal integral curve, then either the intensity on the surface is
not smooth, or the surface is not smooth, or the surface is occluded.
\end{cor}
\begin{proof}
This is the contrapositive of Theorem \ref{thm:foliationthm}, or
part (iii) of the previous corollary. 
\end{proof}
In other words, we have obtained a sufficient condition for a singularity
in intensity, or the surface or an occlusion. 

The following lemma gives us a means for finding integral curves in
closed form. Its relevance to vision was observed by Faugeras (see
pg112 of \cite{faugeras93}); those image curves whose normals are
parallel to the gradient are the level curves of intensity. What follows
is an alternate statement. 

\begin{lem}
\label{lem:faugeras}The Hamiltonian operator $H=\left(-\frac{\partial}{\partial y},\frac{\partial}{\partial x}\right)$
has integral curves $I\left(x,y\right)=constant$
\end{lem}
\begin{proof}
Let $x_{0}$ be the point of interest. The level set going through
$x_{0}$ is ${x\mid f(x)=f(x_{0})}$. Consider a curve $\gamma(t)$
in the level set going through $x_{0}$, so we will assume that $\gamma(0)=x_{0}$.
We have

\[
f((\gamma(t))=f(x_{0})=c\]

Now let us differentiate at $t=0$ by using the chain rule. We find

\[
J_{f}(\mathbf{x}_{0})\gamma^{\prime}(0)=0\]

Equivalently, the Jacobian of $f$ at $x_{0}$ is the gradient at
$x_{0}$

\[
\nabla f(x_{0})\cdot\gamma^{\prime}(0)=0\]

Thus, the gradient of f at $x_{0}$is perpendicular to the tangent
$\gamma^{\prime}(0)$ to the curve (and to the level set) at that
point. Since the curve $\gamma(t)$ is arbitrary, it follows that
the gradient is perpendicular to the level set. 

Thus the integral curves of the Hamiltonian are level curves of intensity,
that is curves with $I\left(x,y\right)={\rm constant}$. 
\end{proof}
Thus the Hamiltonian is an operator whose integral curves we can calculate
in closed form, circumventing all the errors associated with numerical
differentiation, numerical equation solving, and even floating point
arithmetic. However there is an exception at a finite number of points
where the vector field on a surface vanishes. 

\label{LevelCurves}%
\begin{figure}
\caption{Level curves of smooth $I$ are boundaries of $I(p)<c$}
\includegraphics{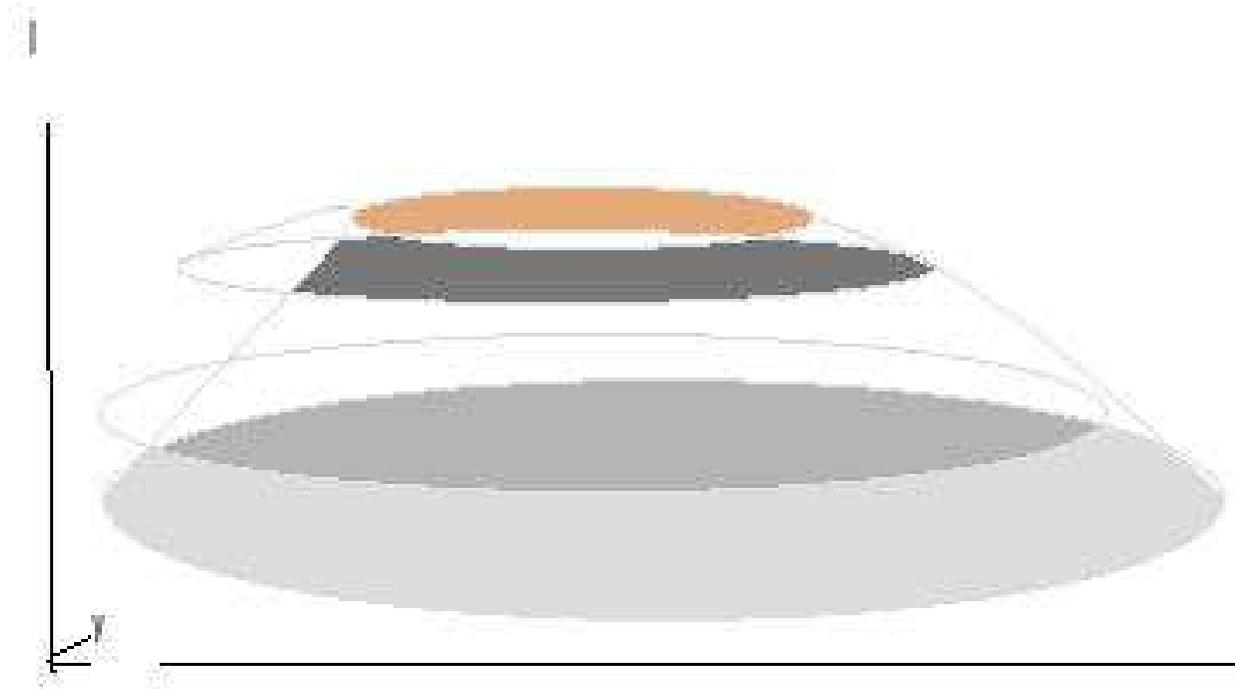}

\end{figure}

An image is only a discrete sample of measurements of light reflected
from a surface. If a pixel has a particular measurement value, none
of its neighbors necessarily has the same value. It is therefore not
feasible to calculate level curves by trying to chain together pixels
that have the same measurement value. A more sensible procedure is
to calculate the set of image points whose measurement value is less
than the particular measurement value, and then compute the boundary
of this set (Figure \ref{LevelCurves}).

\begin{defn}
Let \[
I_{c}=\partial\left\{ p\mid I(p)<c\right\} \]
where $\partial$ denotes the boundary of the set.
\end{defn}
\begin{thm}
\label{thm:discontinuity-existence}If $p\in I_{c}\cap I_{c^{\prime}}$
with $c\neq c^{\prime}$ then either the intensity on the surface
is not smooth, or the surface is not smooth, or the surface is occluded.
\end{thm}
\begin{proof}
I$_{c}$and I$_{c^{\prime}}$ consist of level curves of intensity,
which are maximal integral curves of the Hamiltonian operator (from
Lemma \ref{lem:faugeras}). Since $p$ lies in these two distinct
maximal integral curves, it follows from Corollary \ref{cor:contra-foliation}
that either the intensity on the surface is not smooth, or the surface
is not smooth, or the surface is occluded. 
\end{proof}
The surfaces of objects in scenes have another property, namely they
are bounded, and can be considered to be closed sets. A closed and
bounded subset of ${\bf \real}^{m}$ is called compact. A more general
definition of compactness can be given in the more general setting
of topological spaces. A collection of open sets is said to cover
a set $S$ if $S$ is contained in the union of all the open sets
in the collection. A set $S$ is called compact\index{compact} if
every collection of open sets covering $S$ has a finite sub-collection
that covers $S$. This is not an easy property to grasp, but it is
a very important one. Theorem 6.3 in \cite{munkres75} asserts that
the closed and bounded subsets of ${\bf \real}^{m}$ coincide with
the compact subsets.

Surfaces can therefore be considered to be compact manifolds, but
this includes parts of the surface that are not visible, or in contact
with other surfaces like the ground.

It can be shown that all maximal integral curves of a compact smooth
manifold are complete \cite{warner83}. Thus a hypothesis that an
image is a smooth image of an unoccluded, smooth compact surface implies
that all maximal integral curves are closed curves. We intuitively
know that this follows from an absurd assumption, namely that \textit{all}
of a surface is visible. However, since we want to reject such a hypothesis,
we must employ the consequences which are most certain to lead to
a contradiction. This is why we will generate only closed curves,
by completion of curves if necessary. The phenomenon of completion
occurs in human vision, and can be demonstrated vividly by an optical
illusion called the Kanisza triangle.

An algorithm was implemented based on Theorem \ref{thm:discontinuity-existence},
using values $c$ in constant increments. This algorithm used no floating
point arithmetic. This was run on a super-computer belonging to the
University of Melbourne. The results of this were surprising. If very
small increments are chosen then almost the entire image is returned.
What this shows is that the algorithm can detect discontinuities that
are dense, not just discontinuities that fall on a curve. However
if larger increments are chosen, then edges are returned. The discontinuities
found do not surround objects completely.

\section{Relation to threshold's}

We will now show that the existence of a neighboring point and a threshold
such that the difference in measurements exceeds the threshold is
a necessary but not sufficient condition for an occlusion or singularity
to exist.

\begin{lem}
If $p\in I_{c}\cap I_{c^{\prime}}$then there is a neighboring point
$q$ such that \[
I\left(p\right)-I\left(q\right)>\left|c-c^{\prime}\right|\]

\end{lem}
\begin{proof}
If $p\in I_{c}\cap I_{c^{\prime}}$then $p\in\partial\left\{ r\mid I(r)<c\right\} $and
$p\in\partial\left\{ r\mid I\left(r\right)<c^{\prime}\right\} $

Thus $I\left(p\right)\geq c$ and $I\left(p_{c}\right)<c$ for some
neighboring point $p_{c}$. 

Also $I\left(p\right)\geq c^{\prime}$and $I\left(p_{c^{\prime}}\right)<c^{\prime}$for
some neighboring point $p_{c^{\prime}}$. 

Thus $I\left(p\right)-I\left(p_{c^{\prime}}\right)>c-c^{\prime}$
and $I\left(p\right)-I\left(p_{c}\right)>c^{\prime}-c$. 

It follows that $p$ has a neighboring point $q$ such that $I\left(p\right)-I\left(q\right)>\mid c-c^{\prime}\mid$ 
\end{proof}
We have shown that if the sufficient condition for an occlusion or
discontinuity in an image or scene is met at a point in an image (see
Theorem 3.3 in PhD thesis), then there is a neighboring point and
a threshold that must be exceeded. In other words we have obtained
a necessary condition for a sufficient condition. Unfortunately this
does not suffice because the condition is not necessary and sufficient.

\begin{example}
$I\left(p\right)-I\left(q\right)>\mid c-c^{\prime}\mid$does not imply
$p\in I_{c}\cap I_{c^{\prime}}$ Let $I\left(p\right)=5$ $I\left(p_{c}\right)=3$
$I\left(p_{c^{\prime}}\right)=3$ $c=2$ $c^{\prime}=1$ Thus $I\left(p\right)-I\left(p_{c^{\prime}}\right)>c-c^{\prime}$
and $I\left(p\right)-I\left(p_{c}\right)>c^{\prime}-c$ However $I\left(p_{c}\right)\geq c$
so $p$ is not in $I_{c}$ 
\end{example}
What this lemma and example tells us is that threshold's will give
us a set of occlusions and discontinuities, but will also include
points that are neither occlusions nor discontinuities. We now give
a stronger condition that is more useful.

\begin{lem}
Lemma $p\in I_{c}\cap I_{c^{\prime}}$ with $c\neq c^{\prime}$ if
and only if at point $p$ there exist two neighboring points $p^{\prime}$
and $p^{\prime\prime}$ and two thresholds $\varepsilon>0$ and $\varepsilon^{\prime}>0$
such that $I\left(p\right)-I\left(p^{\prime}\right)\geq\varepsilon$
and $I\left(p\right)-I\left(p^{\prime\prime}\right)\geq\varepsilon^{\prime}$
and $\varepsilon+I\left(p^{\prime}\right)\neq\varepsilon^{\prime}+I\left(p^{\prime\prime}\right)$
\end{lem}
\begin{proof}
Suppose at point $p$ there exist two neighboring points $p^{\prime}$
and $p^{\prime\prime}$and thresholds $\varepsilon,\varepsilon^{\prime}$
such that $I\left(p\right)-I\left(p^{\prime}\right)\geq\varepsilon>0$
and $I(p)-I\left(p^{\prime\prime}\right)\geq\varepsilon^{\prime}>0$. 

Let $c=\varepsilon+I\left(p^{\prime}\right)$ 

Then $I\left(p\right)\geq\varepsilon+I\left(p^{\prime}\right)=c$ 

And $c>I\left(p^{\prime}\right)$ since $\varepsilon>0$ 

Thus $I\left(p\right)\geq c$ and $I\left(p^{\prime}\right)<c$. It
follows that $p\in I_{c}=\partial\left\{ p\mid I\left(p\right)<c\right\} $. 

Letting $c^{\prime}=\varepsilon^{\prime}+I(p^{\prime\prime})$ it
follows similarly that $p\in I_{c^{\prime}}$ 

From $c=\varepsilon+I\left(p^{\prime}\right)\neq\varepsilon^{\prime}+I\left(p^{\prime\prime}\right)=c^{\prime}$
it follows $c\neq c^{\prime}$. Hence $p\in I_{c}\cap I_{c^{\prime}}$
with $c\neq c^{\prime}$.
\end{proof}
Suppose $p\in I_{c}\cap I_{c^{\prime}}$ with $c\neq c^{\prime}$.
Then

\[
I\left(p\right)\geq c>I\left(p_{c}\right)\Rightarrow I\left(p\right)-I\left(p_{c}\right)\geq c-I\left(p_{c}\right)>0\]

and

\[
I\left(p\right)\geq c^{\prime}>I\left(p_{c^{\prime}}\right)\Rightarrow I\left(p\right)-I\left(p_{c^{\prime}}\right)\geq c^{\prime}-I\left(p_{c^{\prime}}\right)>0\]

Letting $\varepsilon=c-I\left(p_{c}\right)$ and $\varepsilon^{\prime}=c^{\prime}-I\left(p_{c^{\prime}}\right)$
and $p^{\prime}=p_{c}$ and $p^{\prime\prime}=p_{c^{\prime}}$ it
follows that $I\left(p\right)-I\left(p^{\prime}\right)\geq\varepsilon>0$
and $I\left(p\right)-I\left(p^{\prime\prime}\right)\geq\varepsilon^{\prime}>0$.

The reader should note that the neighboring points $p^{\prime}$ and
$p^{\prime\prime}$ can coincide, providing $\varepsilon\neq\varepsilon^{\prime}$

Consider $H=\left(-\frac{\partial}{\partial t},\frac{\partial}{\partial x}\right)$
which is a Hamiltonian operator. The integral curves of $H$ are level
curves $I\left(x,t\right)$. 

The lemma shows that if $\left(x,t\right)$ has two neighboring points
$\left(x_{i},t_{i}\right)i=1,2$ such that $I\left(x,t\right)-I\left(x_{i},t_{i}\right)>\varepsilon_{i}>0$
with $\varepsilon_{1}+I\left(x_{1},t_{1}\right)\neq\varepsilon_{2}+I\left(x_{2},t_{2}\right)$
then there is an occlusion or discontinuity at$\left(x,t\right)$. 

So if we used two consecutive video frames and two different thresholds,
we could apply the test to the time increments to determine a set
of occlusions and discontinuities. We could also use three consecutive
video frames and two different thresholds, and apply the test. 

Note that we can apply any smooth differential operator to the image
function. For example we could differentiate the green component and
set $I\left(x,y\right)=\left(\frac{\partial G}{\partial x}\right)\left(x,y\right)$.
Then applying the same algorithm to I would determine a set of occlusions
and discontinuities. We can also use higher order smooth differential
operators, or finite difference approximations to smooth differential
operators. In fact we can use absolutely any process that transforms
an arbitrary smooth function into another arbitrary smooth function,
to substitute in a new I into the algorithm.

\section{applications to image analysis or computer vision}

Two main components of a camera or the eye are a lens and a collection
of light sensors on a surface, called the  retina. The retina of a
camera is usually planar, and will be referred to as the retinal plane.
The camera forms an image of the scene in front of it. The correspondence
between a point in the scene and a point in the retinal plane can
be approximately given by a straight line through the optical centre
of the  camera (see figure \ref{perspective} ).

This model can be justified by geometric optics. A convex lens with
two spherical surfaces whose thickness  can be neglected is called
a thin lens in geometric optics. The lens of a camera is usually a
series of thin lenses (called a  thick lens in geometric optics,  see
page 25 of \cite{welford88}). The paraxial or Gaussian assumption
of geometric optics is  that the angle between a scene point and the
optical axis of  the lens system is small. It is known in geometric
optics that, under the paraxial assumption,  the vertical distance
between the scene point and the optical axis and the vertical distance
between its focussed image point and the  optical axis is a constant,
called the magnification of  the lens system (see page 28 of \cite{welford88}).

The elementary physical concept of reflection is that a light ray
incident to a surface will be reflected at an angle to the surface
normal equal to the angle of incidence, such that incident and exitant
rays lie in the same plane through the normal vector. A mirror is
an example of such an ideal, or specular reflector. The microscopic
structure of most materials in scenes  is not smooth, so a scene surface
will scatter light in different directions. Most materials are also
not perfectly homogeneous on a microscopic scale, and thus scatter
light rays that penetrate the surface by refraction and reflection
at boundaries between regions of different refractive indices. Scattered
rays may reemerge near the point of entry, and so contribute to diffuse
reflection. Snow and layers of white paint are examples of materials
with this behaviour \cite{horn-sjoberg89}.

A surface is called Lambertian if it appears equally bright from all
directions, regardless of how it is irradiated, and reflects all incident
light \cite{horn-sjoberg89}. This does not imply that different points
on the same surface have the same intensity, because the value of
the intensity depends on other variables such as the angle between
the surface normal and a light source. When a smooth lambertian surface
is rendered by computer  graphics, its appearance is like a dull smooth
plastic \cite{lorig86}. 

Most materials in scenes are neither specular reflectors nor Lambertian
reflectors, but a combination of both. A surface rendered by computer
graphics will appear shinier as the specular component is increased
\cite{lorig86,hearn-baker86}. The appearance of texture (to the eye)
is due to particularly rapid  variations in normal vectors over small
regions. For example, a random pattern of normal vectors when rendered
by computer graphics can produce an effect of fog. Various repetitive
patterns of normal vectors can be used to render real materials, such
as bricks by computer graphics \cite{hearn-baker86}. The Lambertian
assumption can also be used to directly recover surface shape from
image shading \cite{brooks-chojnacki94,chojnacki-etal94}.

Owing to the microstructure of a scene surface, the light radiated
from the surface undergoes extremely rapid intensity changes over
small regions. This can be confirmed by examining gray values of neighbouring
points in a real image.

\label{perspective}%
\begin{figure}
\caption{}

\includegraphics{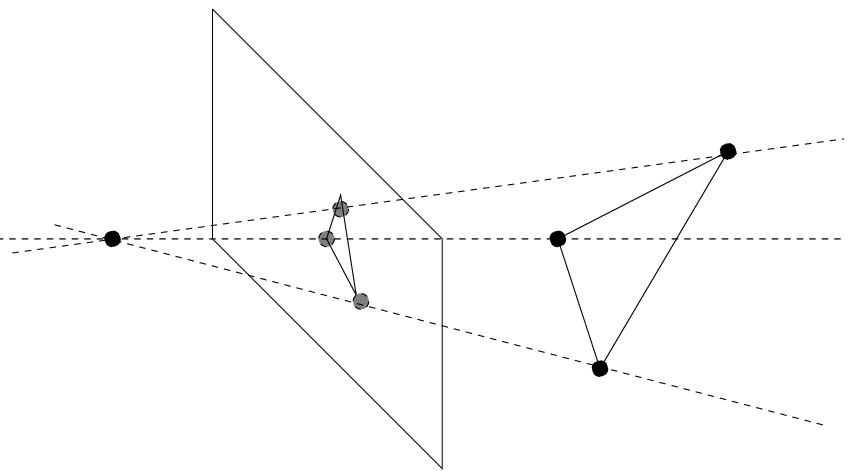}
\end{figure}

Let $\iota:(x,y,z)\rightarrow(t,s)$ where $t=\frac{x}{z},\, s=\frac{y}{z}$,
and $z\neq0$. For a camera with retina at $z=1$ and optical centre
at the origin, the function $\iota$ represents image formation. The
point $(x,y,z)$ is being viewed, and its image is formed at $(\frac{x}{z},\frac{y}{z})$. 

Note that the co-ordinates $(t,s)$ must be $(0,0)$ at the point
where the optical axis meets the retinal plane perpendicularly. It
is preferred to find this point in the image by some procedure. Zooming
the camera and finding the point around which the image expands or
contracts may be one way. The unit of distance is approximately the
distance between the pupil and the ccd.

For the purpose of stereo or binocular vision, points in an image
where the surface being viewed is not smooth need to be distinguished
from occluding points. This is because under the assumption of Lambertian
shading, different viewpoints of the same point have similar intensity
values. Points where the surface being viewed is not smooth possess
invariance to viewpoint. On the other hand occluding points move on
the surface when they are viewed from a different viewpoint. Occluding
points can be eliminated by filtering out zeroes of the Hamiltonian.
This does require floating point arithmetic. For other purposes such
as robot navigation where we are just looking for empty spaces this
is not an issue. 

Also note that we could use any discrete topology on the plane to
define the neighbors of a point. For example we could use 4 neighbors
north, west, east and south of each pixel. Or we could use 8 neighbors.
Or some other topology would also do.

\bibliographystyle{plain}
\bibliography{BASE,PHYSIOL}

\begin{thebibliography}{10}

\bibitem{bhav96}
B.~Bhavnagri.
\newblock {\em Computer Vision using Shape Spaces}.
\newblock PhD thesis, University of Adelaide, 1996.

\bibitem{binmore82}
K.G. Binmore.
\newblock {\em Mathematical analysis: a straightforward approach}.
\newblock Cambridge University Press, 1982.

\bibitem{brickell-clark70}
F.~Brickell and R.S. Clark.
\newblock {\em Differentiable manifolds: an introduction}.
\newblock Van Nostrand Reinhold, 1970.

\bibitem{brooks-chojnacki94}
M.J. Brooks and W.~Chojnacki.
\newblock Direct computation of shape from shading.
\newblock {\em Pattern Recognition}, 1:114--119, 1994.

\bibitem{chojnacki-etal94}
W.~Chojnacki, M.J. Brooks, and D.~Gibbons.
\newblock Revisiting pentlands estimator of light source direction.
\newblock {\em Journal of the Optical Society of America A}, 11(1):118--124,
  1994.

\bibitem{faugeras93}
O.~Faugeras.
\newblock {\em Three dimensional computer vision: a geometric viewpoint}.
\newblock MIT Press, 1993.

\bibitem{hearn-baker86}
D.~Hearn and M.P. Baker.
\newblock {\em Computer Graphics}.
\newblock Prentice-Hall, 1986.

\bibitem{hoffman66}
W.C. Hoffman.
\newblock The lie algebra of visual perception.
\newblock {\em Journal of Mathematical Psychology}, 3(1):65--98, 1966.

\bibitem{horn-sjoberg89}
B.K.P. Horn and R.W. Sjoberg.
\newblock Calculating the reflectance map.
\newblock In B.K.P. Horn and M.J. Brooks, editors, {\em Shape from shading},
  chapter~8, pages 215--244. MIT Press, 1989.

\bibitem{kanizsa79}
G.~Kanizsa.
\newblock {\em Organization in vision, essays on gestalt perception}.
\newblock Praeger, 1979.

\bibitem{le-bhavnagri97}
H.~Le and B.~Bhavnagri.
\newblock On simplifying shapes by subjecting them to collinearity constraints.
\newblock {\em Mathematical Proceedings of the Cambridge Philosophical
  Society}, 121(2), 1997.

\bibitem{lorig86}
G.~Lorig.
\newblock Advanced image synthesis - shading.
\newblock In {\em Advances in computer graphics I}, volume XII of {\em
  Eurographic seminars}, pages 441--456. Springer-Verlag, 1986.

\bibitem{munkres75}
J.R. Munkres.
\newblock {\em Topology: A first course}.
\newblock Prentice-Hall, Englewood Cliffs New Jersey, 1975.

\bibitem{warner83}
F.~W. Warner.
\newblock {\em Foundations of differential manifolds and lie groups}.
\newblock Springer-Verlag, New York, 1983.

\bibitem{welford88}
W.T. Welford.
\newblock {\em Optics}.
\newblock Oxford Physics Series. Oxford Science Publications, 1988.

\bibitem{wertheimer12}
M.~Wertheimer.
\newblock Experimentelle studien uber das sehen von bewegung.
\newblock {\em Zeitschrift f. Psychol.}, 61:161--265, 1912.

\bibitem{wertheimer38}
M.~Wertheimer.
\newblock {\em Laws of organization in perceptual forms}.
\newblock Harcourtm Brace and Co, 1938.

\end{thebibliography}
\end{document}